\documentclass{article}
\usepackage{arxiv}
\usepackage{graphicx}
\usepackage[utf8]{inputenc} %
\usepackage[T1]{fontenc}    %
\usepackage{hyperref}       %
\usepackage{url}            %
\usepackage{booktabs}       %
\usepackage{amsfonts}       %
\usepackage{nicefrac}       %
\usepackage{microtype}      %
\usepackage{lipsum}
\usepackage[]{algorithm2e}

\usepackage{cite}
\usepackage{amsmath}
\usepackage{amsthm}
\usepackage[toc,page]{appendix}

\newtheorem{theorem}{Theorem}
\newtheorem{lemma}{Lemma}

\usepackage{mathptmx}      %
\begin{document}

\title{Extended Regularized Dual Averaging Methods for Stochastic Optimization
}

\author{Jonathan W. Siegel \\
  Department of Mathematics\\
  Pennsylvania State University\\
  University Park, PA 16802 \\
  \texttt{jus1949@psu.edu} \\
  \And Jinchao Xu \\
  Department of Mathematics\\
  Pennsylvania State University\\
  University Park, PA 16802 \\
  \texttt{jxx1@psu.edu} \\
}

\maketitle

\begin{abstract}
We introduce a new algorithm, extended regularized dual averaging (XRDA), for solving regularized stochastic optimization problems, which generalizes the regularized dual averaging (RDA) method. The main novelty of the method is that it allows a flexible control of the backward step size. For instance, the backward step size used in RDA grows without bound, while for XRDA the backward step size can be kept bounded. We demonstrate experimentally that additional control over the backward step size can significantly improve the convergence rate of the algorithm while preserving desired properties of the iterates, such as sparsity. Theoretically, we show that the XRDA method achieves the same convergence rate as RDA for general convex objectives.
\end{abstract}

\section{Introduction}
Optimizing convex objectives with a regularization term which promotes a certain structure of the minimizer, for example the $\ell^1$-norm promoting sparsity or the nuclear norm promoting low rank \cite{candes2009exact}, are very common and important in machine learning. Stochastic optimization of such objectives poses a unique challenge since traditional methods such as stochastic gradient descent (SGD) \cite{lecun1998gradient} are often not effective at producing the desired structure (i.e. sparsity or low-rank) of the iterates \cite{xiao2010dual}. This motivated the introduction of the well-known RDA method and its variants, which are able to effectively produce the desired structure of the iterates \cite{xiao2010dual,chen2012optimal}. We introduce a generalization of RDA, called extended regularized dual averaging (XRDA), which we show can significantly improve convergence while still preserving the desired structure of the iterates.

Let us begin by describing these issues in detail and giving an overview of RDA. Consider the subgradient descent and dual averaging methods \cite{nesterov2009primal} for minimizing a Lipschitz convex function $F$, given by
\begin{equation}\label{eq-1}
 x_{n+1} = x_n - s_ng_n,
\end{equation}
where $g_n\in \partial F(x_n)$. It is well known that with a step size $s_n = n^{-\frac{1}{2}}$ this method attains a convergence rate of $O(n^{-\frac{1}{2}}\log{n})$. The simple dual averaging (SDA) method of Nesterov \cite{nesterov2009primal}, which is given by
\begin{equation}
 x_{n+1} = \arg\min_x\left(\sum_{i=1}^n \langle s_ig_i, x\rangle + \frac{\alpha_{n+1}}{2}\|x - x_1\|^2\right),
\end{equation}
generalizes subgradient descent (note that setting $\alpha_n = 1$ recovers the iteration \eqref{eq-1}). The advantage of the more general SDA method is that by setting $s_n = 1$ and $\alpha_n = \sqrt{n}$, the logarithmic factor in the convergence rate can be removed \cite{nesterov2009primal}. To illustrate the relationship between this new method and original subgradient descent, we rewrite this new iteration as
\begin{equation}
 x_{n+1} = \frac{\alpha_{n}}{\alpha_{n+1}}x_n + \left(1 - \frac{\alpha_{n}}{\alpha_{n+1}}\right)x_1 - \tilde{s}_n g_n,
\end{equation}
where the effective step size is $\tilde{s}_n = \frac{s_n}{\alpha_{n+1}}$. Thus the difference between this new method and the subgradient descent method is an averaging with the initial iterate $x_1$. It is remarkable that this provides a significant improvement in the convergence rate. These results hold in more generality with the $\ell^2$ distance replaced by the Bregman distance $D_{\phi}(x,y)$ with respect to a convex function $\phi$, which we describe in more detail in Section \ref{notation-setup}, but for simplicity we stay with the current setting throughout the introduction.

Next, we consider the composite (or regularized) optimization problem
\begin{equation}\label{split_problem}
    \arg\min_{x\in A}~[f(x) = F(x) + G(x)],
\end{equation}
where $F(x)$ is a convex Lipschitz function and $G(x)$ is a convex function.
A standard method for solving this is the forward-backward subgradient method
\begin{equation}\label{fb_discretization}
\begin{split}
    &x_{n+\frac{1}{2}} = x_{n} - s_ng_n\\
    &x_{n+1} - x_{n+\frac{1}{2}}\in - s_n\partial G(x_{n+1}).
\end{split}
\end{equation}
where $g_n\in \partial F(x_n)$ \cite{bertsekas1997nonlinear}.
Note that the second step above corresponds to backward Euler and is known as the proximal map for $G$ \cite{brezis1978produits}. With a choice of step size $s_n = O(n^{-\frac{1}{2}})$, this method also achieves a convergence rate of $O(n^{-\frac{1}{2}}\log{n})$. As before the logarithm can be removed by introducing a similar averaging with $x_1$ as in the SDA method. Note also that the constants in the convergence rates only depend upon the Lipschitz constant of $F$ and not $G$, which is the advantage of using the forward-backward splitting.
%
%
%
%
%

%

%

%
%
%
%
%
%
%
%

In many cases of practical interest, the subgradients $g_i\in \partial F(x_n)$ in
the forward step are not computed exactly, but rather replaced by an unbiased sample $\tilde g_i$. Using this sample in \eqref{fb_discretization} results in forward-backward stochastic gradient descent. Stochastic gradient descent has proven extremely useful for training a variety of machine learning models 
\cite{lecun1998gradient,lecun1999object}. However, for problems where the minimizer is expected to have a special structure, forward-backward stochastic gradient descent often has the drawback that the iterates it
produces do not have the desired structure. A very common example is sparsity. For instance, consider the forward-backward stochastic
gradient descent algorithm applied to an objective $F(x) + G(x)$ with $G(x) = \lambda\|x\|_1$ an $l^1$ regularization term:
\begin{equation}
\begin{split}
    &x_{n+\frac{1}{2}} = x_{n} - s_n\tilde g_i\\
    &x_{n+1} 
    = \arg\min_x~\left(\lambda s_n\|x\|_1 + \frac{1}{2}\|x - x_{n+\frac{1}{2}}\|^2\right) = \begin{cases} 
      x_{n+\frac{1}{2}} - \lambda s_n & x > \lambda s_n \\
      0 & |x|\leq \lambda s_n \\
      x_{n+\frac{1}{2}} + \lambda s_n & x < -\lambda s_n.
   \end{cases}
\end{split}
\end{equation}
Here $\tilde g_i$ denotes a sampled subgradient. 
Notice that the soft-thresholding parameter in the above iteration is $\lambda s_n$, which must be very small since
the step size $s_n$ must go to $0$ if we wish to converge to the exact optimizer \cite{xiao2010dual} (note here that the gradient samples are stochastic). This means that the iterates generated
by this methods will not be sparse even if the true optimizer is sparse. Further, due to the slow convergence rate of the algorithm, thresholding after training may introduce significant errors even though the true optimizer is sparse \cite{xiao2010dual}. This property of stochastic gradient descent also applies to other types of structure, such as low rank structure expected when using the nuclear norm as a regularizer \cite{candes2009exact}. 

This motivated the introduction of the regularized dual averaging (RDA) method \cite{xiao2010dual}, which is obtained by modifying the SDA algorithm in the following way
\begin{equation}
  x_{n+1} = \arg\min_x\left(\sum_{i=1}^n \langle s_ig_i, x\rangle + \frac{\alpha_{n+1}}{2}\|x - x_1\|^2 + \gamma_nG(x)\right),
\end{equation}
where $\gamma_n = \sum_{i=1}^n s_i$. We can clarify the relationship between this method and the forward-backward subgradient descent method \eqref{fb_discretization} by setting $\alpha_n = 1$ and rewriting it in the following `leap-frog' form
\begin{equation}\label{leap_frog}
     \begin{split}
    &x_{n+\frac{1}{2}} = x_{n - \frac{1}{2}} - s_n\tilde g_i\\
    &x_{n+1} = \arg\min_x~\left( \gamma_nG(x) + \frac{1}{2}\|x - x_{n+\frac{1}{2}}\|^2\right),
\end{split}
\end{equation}
where $\gamma_n = \sum_{i=1}^n s_i$. This method achieves a convergence rate of $O(n^{-\frac{1}{2}}\log{n})$ and the logarithm here can be removed by introducing an averaging with $x_1$ similar to SDA (this corresponds to a different choice of the parameters $s_n$ and $\alpha_n$). Notice that in the RDA iteration \eqref{leap_frog}, the backward step size grows as the sum of the previous forward step sizes $s_n$ and grows without bound as $n\rightarrow \infty$. 
This is in contrast to forward-backward stochastic gradient descent \eqref{fb_discretization}, where the backward step and forward step are the same, and thus the backward step size goes to $0$.
As a result, the RDA method produces sparse iterates while forward backward SGD does not. However, it is quite unintuitive that the backward step size should grow without bound, which we attempt to address in this paper.

We introduce the extended regularized dual averaging (XRDA) method, which generalize the RDA method. This method gives more precise control over the backward step size, for instance a special case of XRDA is the `averaged leap-frog scheme' 
\begin{equation}\label{averaged_leap_frog}
     \begin{split}
    &x_{n+\frac{1}{2}} = \left(1 - \mu_n\right)x_{n - \frac{1}{2}} + \mu_nx_n - s_n\tilde g_i\\
    &x_{n+1} = \arg\min_x~\left(\lambda \gamma_n|x|_1 + \frac{1}{2}\|x - x_{n+\frac{1}{2}}\|^2\right)
\end{split}
\end{equation}
where $0\leq \mu_i\leq 1$ and $\gamma_n = (1-\mu_n)\gamma_{n-1} + s_n$. As we show in Section \ref{XRDA-section}, the convergence rate of this algorithm is $O(n^{-\frac{1}{2}}\log{n})$ and a slightly modified version achieves the optimal black-box convergence rate of $O(n^{-\frac{1}{2}})$ \cite{nemirovsky1983problem,nesterov2013introductory,xiao2010dual}. Further, in Section \ref{experimental-results-section} we show experimentally that choosing the parameter $\mu_n$ so that the backward step $\gamma_n$ remains bounded leads to significantly improved convergence over RDA while preserving the sparsity of the iterates.

Concerning non-convex applications of the XRDA method, we note that the BinaryConnect \cite{courbariaux2015binaryconnect} and BinaryRelax \cite{yin2018binaryrelax} methods, which have been used
to train quantized neural networks, are simply the XRDA method applied to non-convex objectives. Specifically, BinaryConnect consists of the leap frog iteration \eqref{leap_frog} applied to the case where $G = \chi_Q$ is the indicator function of the set $Q$ of quantized parameters, while BinaryRelax is just \eqref{leap_frog} applied to $G(x) = \text{dist}(x,Q)^2$. In addition, the RDA method has also been successfully applied to the training of sparse neural networks \cite{jia2018modified} and recent numerical experiments show that the XRDA method improves upon this for the training of sparse neural networks \cite{siegel2020training}.

The paper is organized as follows. In Section \ref{notation-setup}, we introduce some basic notation and the basic problem formulation. In Section \ref{RDA-section}, we recall the basic ideas behind the dual averaging \cite{nesterov2009primal} and RDA \cite{xiao2010dual} methods.
Then, in Section \ref{XRDA-section}, we introduce the extended regularized dual averaging (XRDA) method and derive its convergence rate. Here we also exhibit a variety of
special cases of the XRDA method and compare with RDA and SDA. Then, in Section \ref{experimental-results-section} we present experimental results showing that the XRDA method significantly exhibits significantly improved convergence relative to RDA while still preserving the sparsity of the iterates. Finally, we give concluding remarks.

\section{Notation and Setup}\label{notation-setup}
In this section, we introduce the basic setting and notation that we will use. This section is essentially a review of the basics of mirror descent \cite{nemirovski2004prox} and is a common framework for subgradient methods in convex optimization \cite{bubeck2015convex}. We refer to \cite{rockafellar2015convex} for the basic notions of relative interior, lower semi-continuity, and subgradients of convex functions.

We are interested in first-order methods for optimizing convex functions $f$ on a closed convex subset $A\subset V$ of a Banach space $V$. A universal issue, which applies to all first-order methods, including gradient descent, subgradient descent, accelerated gradient descent, and conjugate gradient, is that gradients and subgradients of the objective lie naturally in the dual space $V^*$, i.e.
\begin{equation}
 \nabla f\in V^*,~\partial f\subset V^*.
\end{equation}
Consequently, it does not a priori make sense to add gradients of $f$ to primal points in $V$. In order to circumvent this, we will introduce a convex lower semicontinuous function $\phi:V\rightarrow \mathbb{R}\cup \{\infty\}$, called a mirror function (see \cite{bubeck2015convex,nemirovski2004prox,beck2003mirror,nemirovsky1983problem}). We require that $\phi(x) = \infty$ for $x\notin A$, and that $\phi$ is strongly convex with parameter $\sigma$, i.e. that for $x,y\in V$ and $\alpha\in[0,1]$, we have
\begin{equation}
 \phi(\alpha x + (1-\alpha) y) \leq \alpha\phi(x) + (1-\alpha)\phi(y) - \frac{\sigma}{2}\|x-y\|_V^2.
\end{equation}
Note that this is only non-trivial if $x,y\in A$.
The subdifferential of the mirror function $\phi$ gives a map $(\partial \phi)^{-1}:V^*\rightarrow A$, defined by
\begin{equation}\label{mirror-map}
 (\partial \phi)^{-1}(h) = \arg\min_{x\in V} [\phi(x) - \langle h, x\rangle].
\end{equation}
Our assumptions on $\phi$ imply that the above minimizer will lie in $A$. Further, the strong convexity of $\phi$ implies that the minimizer above is unique. From the objective \eqref{mirror-map}, it is also clear that $x = (\partial \phi)^{-1}(h)$ iff $h\in \partial \phi(x)$.

The idea of mirror descent is to perform gradient descent in the dual space and to use $(\partial \phi)^{-1}$ to move iterates from the dual to the primal space. Specifically, we let $h_0\in V^*$ and set $x_0 = (\partial \phi)^{-1}(h_0)$. The mirror gradient step \cite{nemirovski2004prox,beck2003mirror,nemirovsky1983problem} with step size $s_n$ is given by
\begin{equation}\label{mirror-descent-1}
 h_{n+1} = h_n - s_ng_n,~x_{n+1} = (\partial \phi)^{-1}(h_{n+1}),
\end{equation}
where $g_n\in \partial f(x_n)$. This iteration clearly preserves the property that $h_n\in \partial \phi(x_n)$ at every step. 

The iteration \eqref{mirror-descent-1} is commonly written using the Bregman distance \cite{bregman1967relaxation}
\begin{equation}
 D_{\phi,h}(x,y) = \phi(x) - \phi(y) - \langle h, x-y\rangle,
\end{equation}
defined for $x,y\in V$ and $h\in \partial \phi(x)$, as follows
\begin{equation}\label{mirror-descent-2}
 x_{n+1} = \arg\min_{x\in V} \left(s_n\langle g_n, x\rangle + D_{\phi,h_n}(x,x_n)\right),~h_{n+1} = h_n - s_ng_n.
\end{equation}
This formulation will also be important in our later analysis.

Let us consider a few common examples to illustrate this abstract framework. Suppose that $V$ is a Hilbert space and $A = V$. In this case, a simple choice for $\phi$ is the squared norm $\phi(x) = \frac{1}{2}\|x\|_V^2$. The map $(\partial \phi)^{-1}$ in this case is the identity and the mirror descent iteration \eqref{mirror-descent-1} reduces to regular gradient descent.

A common variation on this example is to choose the mirror function $\phi(x) = \frac{1}{2}\langle x, Ax\rangle_V$ where $A$ is a positive definite operator on the Hilbert space $V$. In this case, the map $(\partial \phi)^{-1}$ is given by $(\partial \phi)^{-1}(x) = A^{-1}x$, and mirror descent \eqref{mirror-descent-1} corresponds to preconditioned gradient descent with (linear) preconditioner $A^{-1}$.

%
%
%

%
%
%
%
%
%
%
%
%
%
%
%
%

%
%
Finally, another common special case\cite{juditsky2008learning,nemirovsky1983problem,ben1999conjugate,bubeck2015convex,boyd2004convex} is when $V = \ell^1_n$ and $A = \{x\geq 0:x^T1 = 1\}$ is the probability simplex. In this case, one chooses 
$$\phi(x) = \begin{cases}
             \displaystyle\sum_{i=1}^n x_i\log(x_i) & x\in A \\
             \infty & \text{otherwise}
            \end{cases}
$$ 
as the negative entropy, which is strongly convex with respect to the $|\cdot|_1$ norm. This fact is known as Pinsker's inequality \cite{pinsker1964information}. In this case, the map $(\partial \phi)^{-1}$ is given by 
$$(\partial \phi)^{-1}(x)_i = \left(\sum_{j=1}^n e^{x_j}\right)^{-1}e^{x_i},$$
which is called the softmax function. See \cite{bubeck2015convex} and the references therein for more examples.

In the following, we will need to consider a special case of this setup where we additionally assume that $\phi$ restricted to the affine hull $\text{aff}(A)$ is differentiable on the relative interior $\text{relint}(A)$ (see \cite{rockafellar2015convex,bertsekas1997nonlinear,zalinescu2002convex}, for instance, for the definition and basic properties of the affine hull and relative interior). In addition, we will need that $\nabla \phi$ (here we are restricting to the affine hull) satisfies $\nabla \phi(x)\rightarrow \infty$ as $x\rightarrow \partial_{ri} A$ approaches the relative boundary of $A$. Note that all of the examples given here satisfy these stronger assumptions.

In fact, in this case it is convenient to make the additional assumption that $\text{aff}(A) = V$. This can be assumed without loss of generality by restricting to $\text{aff}(A)$. These additional assumptions imply that $\phi$ is differentiable on $A^o$ and that $\nabla \phi:A^o\rightarrow V^*$ and $(\nabla \phi)^{-1}:V^*\rightarrow A^o$ (here $A^o$ denotes the interior of $A$) gives inverse bijections between $V^*$ and $A^o$. Additionally, the Bregman distance is given by
\begin{equation}\label{simplified-bregman}
  D_{\phi}(x,y) = \phi(x) - \phi(y) - \langle \nabla \phi, x-y\rangle,
\end{equation}
for $x,y\in A^o$ in the interior of $A$ (see \cite{nemirovski2004prox,bubeck2015convex} for these general properties and theory).
The mirror descent iteration \eqref{mirror-descent-2} simplifies to
\begin{equation}\label{mirror-descent-final}
 x_{n+1} = \arg\min_{x\in V} \left(s_n\langle g_n, x\rangle + D_{\phi}(x,x_n)\right),
\end{equation}
where $g_n\in \partial f(x_n)$. Equivalently, we can write $\nabla \phi(x_{n+1}) = \nabla \phi(x_n) - s_ng_n$.

The iteration \eqref{mirror-descent-final} corresponds to a forward gradient step. Analogously, there is a corresponding backward gradient step for an objective $G(x)$. This corresponds to solving the backward equation
\begin{equation}\label{backward-mirror-descent-1}
 \nabla \phi(x_{n+1}) = \nabla \phi(x_n) - s_ng_{n+1},
\end{equation}
where $g_{n+1} \in \partial G(x_{n+1})$. This is easily seen to be equivalent to 
\begin{equation}\label{backward-mirror-descent-final}
 x_{n+1} = \arg\min_{x\in V} (s_nG(x) + D_{\phi}(x,x_n)).
\end{equation}
%
%

%

%
%
%
%
%

%
%
%
%
%
%
%
%
%
%
%
%
%
%
%
%
%
%
%
%
%
%
%
%
%

For practical implementation, it is important that the optimization problem appearing in \eqref{backward-mirror-descent-final} can be efficiently solved. We note that if $\phi = \frac{1}{2}\|\cdot\|_2^2$, then two common choices for $f$ are
\begin{equation}
 G(x) = i_C(x) = 
 \begin{cases}
                    0 &~ \text{if $x\in C$} \\
		    +\infty &~ \text{if $x\notin C$},
 \end{cases}
\end{equation}
in which case the mirror-prox update \eqref{backward-mirror-descent-final} is a projection onto the set $C$, and $G(x) = |x|_1$, in which case the mirror-prox \eqref{backward-mirror-descent-final} update is the well-known soft-thresholding step \cite{tibshirani1996regression}.

In the following, we will use combinations of the forward \eqref{mirror-descent-final} and backward steps \eqref{backward-mirror-descent-final} to optimize a composite objective $f(x) = F(x) + G(x)$.

\section{Dual Averaging and Regularized Dual Averaging}\label{RDA-section}
We begin by briefly recalling the dual averaging method \cite{nesterov2009primal} for minimizing a
convex objective $f$ to help explain the idea behind both RDA \cite{xiao2010dual} and the XRDA method.

The idea behind the dual averaging method is to observe that, due to the convexity of the objective $f$, the sequence of iterates $x_i$ and 
subgradients $g_i\in \partial f(x_i)$ generated by our algorithm leads to a sequence of lower bounds on the objective
\begin{equation}
 l_i(z) = f(x_i) + \langle g_i, z - x_i\rangle.
\end{equation}
Now let $s_i$ be a sequence of weights and consider the primal average
\begin{equation}
 \bar{x}_n = S_n^{-1}\displaystyle\sum_{i=1}^n s_ix_i,
\end{equation}
where $S_n = \sum_{i=1}^ns_i$. From the convexity of $f$, we obtain a bound on the objective at $\bar{x}_n$
\begin{equation}\label{eq_314}
 f(\bar{x}_n) \leq S_n^{-1}\displaystyle\sum_{i=1}^n s_if(x_i),
\end{equation}
while the corresponding dual average (i.e. average of the lower bounds $l_i$) is a lower bound on $f$, i.e.
\begin{equation}\label{eq_318}
 f(z) \geq S_n^{-1}\displaystyle\sum_{i=1}^n s_il_i(z) = S_n^{-1}\displaystyle\sum_{i=1}^n s_i(f(x_i) + \langle g_i, z - x_i\rangle).
\end{equation}
Combining equations \eqref{eq_314} and \eqref{eq_318} we see that
\begin{equation}\label{eq_309}
 f(\bar{x}_n) - f(z) \leq S_n^{-1}\displaystyle\sum_{i=1}^n s_i(f(x_i) - l_i(z)) = S_n^{-1}\displaystyle\sum_{i=1}^n s_i\langle g_i, x_i - z\rangle
\end{equation}
holds for all $z$, in particular it holds for $z = x^*$. The challenge now is to choose the sequence of iterates $x_i$
and weights $s_i$
so that the right hand side of equation \eqref{eq_309} goes to $0$ as $n\rightarrow \infty$ for $z = x^*$. The solution to this problem presented in the seminal paper \cite{nesterov2009primal}
by Nesterov is to choose
\begin{equation}\label{SDA_iteration}
 x_{n+1} = \arg\min_x \left(\displaystyle\sum_{i = 1}^n s_il_i(x) + \alpha_{n+1}D_\phi(x,x_1)\right)
\end{equation}
for a sequence of parameters $s_i$ and $\alpha_i$, in which case we have
\begin{equation}
 \displaystyle\sum_{i=1}^n \langle g_i, x_i - z\rangle\leq 
 \left(\alpha_nD_\phi(z,x_1) + 
 \frac{M^2}{2\sigma}\displaystyle\sum_{i = 1}^n \frac{s_i^2}{\alpha_i}\right),
\end{equation}
where $M$ is the Lipschitz constant of $f$ (i.e. $\|g_i\|_{V^*}\leq M$). The choice $\alpha_i = 1$ and $s_i = i^{-\frac{1}{2}}$, which corresponds to subgradient descent, gives a convergence rate which is $O(n^{-\frac{1}{2}}\log{n})$. Optimal choices of the weights $s_i$ and
parameters $\alpha_i$ allow one to remove the logarithmic factor and obtain an objective convergence rate which is $O(n^{-\frac{1}{2}})$ \cite{nesterov2009primal}.

Next, we wish to generalize the dual averaging method to composite objectives $f$ which are the sum of two pieces
\begin{equation}\label{composite_problem}
    \arg\min_{x\in A} [f(x) = F(x) + G(x)].
\end{equation}
Here $F$ is a convex function which is Lipschitz with respect to the norm $\|\cdot\|_V$ and
$G$ is a convex function for which we can solve the backward step \eqref{backward-mirror-descent-final}.
As discussed in the previous section, 
common choices for $G$ include, for example, 
the $l_1$-norm (if $\phi = \frac{1}{2}\|x\|_2^2$) or the characteristic function of a convex set.

In his seminal work \cite{xiao2010dual}, Xiao extended Nesterov's analysis to the composite case by using the lower bound 
\begin{equation}\label{composite_lower_bound}
 l_i(z) = F(x_i) + \langle g_i, z - x_i\rangle + G(z) \leq f(z),
\end{equation}
where $g_i\in \partial F(x_i)$, in \eqref{SDA_iteration}. So here we are using a linear lower bound on the piece $F$ and using the function $G$ itself in the lower bound. This results in the RDA method \cite{xiao2010dual}
\begin{equation}\label{naive_iteration}
 x_{n+1} =
 \arg\min_x \left(\displaystyle\sum_{i = 1}^n \langle s_ig_i, x\rangle + \alpha_{n+1}D_\phi(x,x_1) + \gamma_{n+1} G(x)\right).
\end{equation}
where $\gamma_{n + 1} = \sum_{i = 1}^n s_i$. From the properties of the Bregman distance discussed in the preceding section, the optimizer above satisfies
\begin{equation}
 \nabla \phi(x_{n+1}) = \nabla \phi(x_{n+1}) - \frac{1}{\alpha_{n+1}}\displaystyle\sum_{i = 1}^ns_ig_i - \frac{\gamma_n}{\alpha_{n+1}} \tilde{g}_{n+1},
\end{equation} 
where $\tilde{g}_{n+1}\in\partial G(x_{n+1})$.
Introducing the intermediate iterate $x_{n+\frac{1}{2}}$ defined by 
$$\nabla \phi(x_{n+\frac{1}{2}}) = \nabla \phi(x_{n+1}) - \frac{1}{\alpha_{n+1}}\displaystyle\sum_{i = 1}^ns_ig_i,$$
we can rewrite \eqref{naive_iteration} in terms of a forward and backward step as
\begin{equation}
 \begin{split}
  x_{n+\frac{1}{2}} &=
 \arg\min_x \left(\displaystyle\sum_{i = 1}^n \langle s_ig_i, x\rangle + \alpha_{n+1}D_\phi(x,x_1)\right) \\
 x_{n+1} &=
 \arg\min_x \left(\alpha_{n+1}D_\phi(x,x_{n+\frac{1}{2}}) + \gamma_{n+1} G(x)\right).
 \end{split}
\end{equation}

The advantages of this method are twofold. First, the convergence behavior only depends upon the Lipschitz constant of the function $F$. Second, the backward step $\gamma_n$ grows, which allows the method to obtain sparse iterates even with large step sizes and noisy gradients \cite{xiao2010dual}.

\section{Extended Regularized Dual Averaging}\label{XRDA-section}
In this section, we slightly generalize the RDA method to obtain the XRDA method. To do this, we consider the more general lower bound
\begin{equation}\label{XRDA-lower_bound}
    l_i(z) = (F(x_i) + \langle g_i, z - x_i\rangle) + q_i(G(x_i) + \langle h_i, z - x_i\rangle) + (1-q_i)G(z) \leq f(z),
\end{equation}
where $g_i \in \partial F(x_i)$, $h_i\in \partial G(x_i)$, and $q_i \geq 0$ is a parameter. This is easily seen to be a lower bound since $h_i\in \partial G(x_i)$ so that $ G(x_i) + \langle h_i, z - x_i\rangle \leq G(z)$. Note that this lower bound is not even necessarily convex if $q_i > 1$, but remarkably we can still obtain convergence of the resulting method. We plug this lower bound into equation \eqref{SDA_iteration}, to get the method (setting $t_i = s_iq_i$)
\begin{equation}\label{argmin_formulation_generalized_RDA}
    x_{n+1} = \arg\min_x \left(\displaystyle\sum_{i = 1}^n \langle s_ig_i + t_ih_i, x\rangle + \alpha_{n+1}D_\phi(x,x_1) + \gamma_{n+1} G(x)\right).
\end{equation}
Here $\gamma_{n+1} = \sum_{i=1}^n (1-q_i)s_i = \sum_{i=1}^n (s_i - t_i)$ (restrictions introduced later on the parameters will guarantee that $\gamma_{n+1} > 0$). If we introduce the new parameter $\mu_i = t_i/\gamma_i$, we obtain the perhaps more enlightening relation $\gamma_{n+1} = (1 - \mu_n)\gamma_n + s_n$.

Of course, this is only useful if we can determine an $h_i\in \partial G(x_i)$. We note that the optimality condition corresponding to the iteration \eqref{argmin_formulation_generalized_RDA}, combined with the simple form of the Bregman distance \eqref{simplified-bregman}, implies that
\begin{equation}\label{choice-of-h}
 h_{n+1} :=\gamma_{n+1}^{-1}\left(\alpha_{n+1}(\nabla \phi(x_{n+1}) - \nabla \phi(x_1)) + \displaystyle \sum_{i=1}^n s_ig_i + t_ih_i\right) \in \partial G(x_{n+1}),
\end{equation}
provides an element in $\partial G(x_{n+1})$. Furthermore, if we assume that $x_1\in \arg\min_x G(x)$, then we may choose $h_1 = 0$. Utilizing this choice of $h_{n}$ in \eqref{argmin_formulation_generalized_RDA}, we obtain the XRDA method, which we have summarized in forward form in Table \eqref{xrda-method}.

\RestyleAlgo{boxruled}
\begin{table}
{\LinesNumberedHidden
\begin{algorithm}[H]
\caption{The XRDA Method}
\KwData{$f(x) = F(x) + G(x)$ a composite convex function where $F$ is Lipschitz, $\alpha_n$, $\mu_n$, and $s_n$ sequences of parameters, $N$ an iteration count}
$x_1 \gets \text{initial point}$\;
$x_{\frac{1}{2}} \gets x_0$\;
\For {$n=1,...,N$}{
    $\nabla\phi(x_n^\prime) = \left(1 - \mu_n\right) \nabla\phi(x_{n - \frac{1}{2}}) + \mu_n\nabla\phi(x_n)$\;
    $\nabla\phi(x_{n+\frac{1}{2}}) = \frac{\alpha_n}{\alpha_{n+1}}\nabla\phi(x_n^\prime) + 
    \left(1 - \frac{\alpha_n}{\alpha_{n+1}}\right)\nabla\phi(x_1) - \frac{s_n}{\alpha_{n+1}}g_n$ where $g_n\in \partial F(x_n)$\;
    $\nabla\phi(x_{n+1}) \in \nabla\phi(x_{n+\frac{1}{2}}) - \frac{\gamma_{n+1}}{\alpha_{n+1}}\partial G(x_{n+1})$\;
}
\end{algorithm}}
\caption{The XRDA Iteration for Composite Optimization}
\label{xrda-method}
\end{table}

An illustrative special case of the algorithm can be obtained by setting $\alpha_n = 1$, which gives
\begin{equation}\label{generalized_RDA-simpler}
 \begin{split}
    &\nabla\phi(x_{n+\frac{1}{2}}) = \left(1 - \mu_n\right) \nabla\phi(x_{n - \frac{1}{2}}) + \mu_n\nabla\phi(x_n) - s_ng_n\\
    &\nabla\phi(x_{n+1}) \in \nabla\phi(x_{n+\frac{1}{2}}) - \gamma_{n+1}\partial G(x_{n+1}).
\end{split}
\end{equation}
Note that the last line above is a backward proximal step and can also be written as
\begin{equation}
 x_{n+1} = \arg\min_{x\in V} \left(\gamma_{n+1}G(x) + D_{\phi}(x,x_{n+\frac{1}{2}})\right).
\end{equation}
This choice results in a method which is analogous to subgradient descent, and achieves a convergence rate which is $O(n^{-\frac{1}{2}}\log{n})$. Removing the logarithmic factor is possible by taking a more sophisticated choice of the parameter $\alpha_n$, analogous to dual averaging \cite{nesterov2009primal}. Finally, in the special case where $A = V = \mathbb{R}^n$ and $\phi(x) = \frac{1}{2}\|x\|_2^2$, we get the iteration
\begin{equation}\label{generalized_RDA-simplest}
 \begin{split}
    &x_{n+\frac{1}{2}} = \left(1 - \mu_n\right) x_{n - \frac{1}{2}} + \mu_nx_n - s_ng_n\\
    &x_{n+1} = \arg\min_{x\in V} \left(\gamma_{n+1}G(x) + \frac{1}{2}\|x - x_{n+\frac{1}{2}}\|_2^2\right).
\end{split}
\end{equation}
Since $\gamma_{n+1} = (1 - \mu_n)\gamma_n + s_n$, we can control the backward stepsize $\gamma_n$ by choosing the averaging parameter $\mu_n$ appropriately. We give a full comparison of the XRDA method with forward-backward gradient descent, the SDA method and the RDA method in table \ref{xrda-comparison-table}. Here we can see how different parameter choices result in different special cases of the method.

\begin{table}\label{xrda-comparison-table}
\begin{center}
\begin{tabular}{ |p{1.5cm}|p{6.0cm}|p{7.0cm}| } 
 \hline
  & $\mu_n = 0$ & $\mu_n = 1$ \\ 
 \hline
 $\alpha_n = 1$ \newline $s_n = n^{-\frac{1}{2}}$ & Forward-Backward Gradient descent\newline$x_{n+\frac{1}{2}} = x_n - s_ng_n$\newline $x_{n+1} = \arg\min_{x\in V} \frac{1}{2}\|x - x_{n+\frac{1}{2}}\|_V^2 + s_nG(x)$ & Forward-Backward RDA\newline$x_{n+\frac{1}{2}} = x_{n+1} - s_ng_n$\newline $x_{n+1} = \arg\min_{x\in V} \frac{1}{2}\|x - x_{n+\frac{1}{2}}\|_V^2 + \left(\sum_{i=1}^ns_n\right)G(x)$\\
 \hline
 $\alpha_n = \sqrt{n}$ \newline $s_n = 1$ & Simple Dual Averaging (SDA)\newline introduced in \cite{nesterov2009primal} (for $G = 0$) & Regularized Dual Averaging (RDA)\newline
 introduced in \cite{xiao2010dual} \\ 
 \hline
\end{tabular}
\end{center}
\caption{Demonstrates how particular parameter choices in the XRDA algorithm \eqref{xrda-method} correspond to SDA, RDA, and forward-backward gradient descent.}
\end{table}

Similar to the dual averaging and RDA methods, the XRDA method is robust to noisy gradients. This means that in more generality, we can let $g_i$ be independent random variables whose expectation satisfy $\mathbb{E}(g_i)\in \partial F(x_i)$.

Specifically, let $(\Sigma,\mathcal{F},\mathcal{P})$ denote a probability space and define the subgradient sample function (dependent on the
function $F$)
$g:V\times \Sigma\rightarrow V^*$. Given a point $x\in V$ and a random sample $\xi\in \Sigma$, $g$ returns an unbiased
sample of an element in the subgradient of the objective $F$ at $x$. Mathematically, this means that we require
\begin{equation}\label{sample_condition}
 \mathbb{E}_\xi(g(x,\xi))\in \partial F(x).
\end{equation}

Using this in place of $g_i$ in algorithm \eqref{xrda-method}, we obtain a stochastic version of the XRDA method. We have the following convergence theorem.
\begin{theorem} \label{non-weighted-theorem}
 Assume that $F$ is convex and the gradient samples are bounded by $M$, i.e. that $\|g(x,\xi)\|_{V^*}\leq M$. Let $x^*\in \arg\min_x f(x)$. 
 Let $x_n$ be determined by algorithm \eqref{xrda-method}
 with $0\leq t_i\leq \gamma_i$, $\alpha_i\leq \alpha_{i+1}$, and $s_{i+1}\leq s_i$.
 Then, if $x_1\in \arg\min_x G(x)$ and $h_1 = 0$ (i.e. $x_{\frac{1}{2}} = x_1$), we have
 \begin{equation}
  \mathbb{E}(f(\bar{x}_n) - f(x^*)) \leq S_n^{-1} \left(\alpha_nD_\phi(x^*,x_1) + \frac{M^2}{2\sigma}\displaystyle\sum_{i = 1}^n \frac{s_i^2}{\alpha_i}\right),
 \end{equation}
 where $S_n = \sum_{i=1}^ns_i$ and $\bar{x}_n = S_n^{-1}\sum_{i = 1}^n s_ix_i$ is a weighted average of the iterates with weights $s_i$.
 We also have
 \begin{equation}
  \mathbb{E}\left(\min_{i = 1,...,n} f(x_i) - f(x^*)\right) \leq S_n^{-1} \left(\alpha_nD_\phi(x^*,x_1) + \frac{M^2}{2\sigma}\displaystyle\sum_{i = 1}^n \frac{s_i^2}{\alpha_i}\right).
 \end{equation}
\end{theorem}
Notice that $\gamma_1$ is the empty sum and so is $0$. This means since $0\leq t_1\leq \gamma_1$ that $t_1 = 0$ as well. In addition, the condition that $0\leq t_1\leq \gamma_1$ together with the relation $\gamma_{n+1} = \sum_{i=1}^n (s_i - t_i)$ imply that $\gamma_i \geq s_{i-1} > 0$.

 \begin{proof}
  First consider the case where the subgradients are not sampled stochastically, but are deterministic. We begin by noting that by convexity we have
 \begin{equation}\label{eq491}
  f(\bar{x}_n) - f(x^*) \leq S_n^{-1}\displaystyle\sum_{i = 1}^n s_i(f(x_i) - f(x^*)),
 \end{equation}
 and since the minimum is smaller than the average we have
 \begin{equation}\label{eq495}
  \min_{i = 1,...,n} f(x_i) - f(x^*) \leq S_n^{-1}\displaystyle\sum_{i = 1}^n s_i(f(x_i) - f(x^*)).
 \end{equation}
 Thus it suffices to prove that
 \begin{equation}
 \displaystyle\sum_{i = 1}^n s_i(f(x_i) - f(x^*)) \leq 
 \left(\alpha_nD_\phi(x^*,x_1) + \frac{M^2}{2\sigma}\displaystyle\sum_{i = 1}^n \frac{s_i^2}{\alpha_i}\right).
 \end{equation} 
 For this we first use the sub-gradient property (and the assumption that $t_i \geq 0$) to get
 \begin{equation}\label{eq_rda_100}
 \begin{split}
  \displaystyle\sum_{i = 1}^n s_i(f(x_i) - f(x^*)) &\leq \displaystyle\sum_{i = 1}^n \langle s_ig_i + t_ih_i, x_i - x^*\rangle \\
  &+ \displaystyle\sum_{i = 1}^n (s_i - t_i)(G(x_i) - G(x^*)),
  \end{split}
 \end{equation}
 and rewrite the first sum above as
 \begin{equation}\label{eq_rda_105}
 \begin{split}
  \displaystyle\sum_{i = 1}^n \langle s_ig_i + t_ih_i, x_i - x^*\rangle &= \displaystyle\sum_{i = 1}^n \langle s_ig_i + t_ih_i, x_n - x^*\rangle \\
  &+ \displaystyle\sum_{i = 1}^{n-1} \displaystyle\sum_{j = 1}^i\langle s_jg_j + t_jh_j, x_i - x_{i+1}\rangle.
  \end{split}
 \end{equation}
 We proceed by bounding the term
 \begin{equation}
  \sum_{j = 1}^i\langle s_jg_j + t_jh_j, x_i - z\rangle.
 \end{equation}
 Recall from \eqref{argmin_formulation_generalized_RDA} that $x_i = \arg\min_x f_i(x)$ where
 \begin{equation}
  f_i(x) = \displaystyle\sum_{j = 1}^{i-1} \langle s_jg_j + t_jh_j, x\rangle + \alpha_{i}D_\phi(x,x_1) + \gamma_i G(x_i).
 \end{equation}
 The definition of $h_i\in \partial G(x_i)$ given in \eqref{choice-of-h} (note that $i = 1$ is not a problem since 
 $x_1\in \arg\min_x G(x)$ and $h_1 = 0$) means that
 \begin{equation}
  0 = \alpha_i (\nabla \phi(x_i) - \nabla \phi(x_1)) + \gamma_i h_i + \displaystyle\sum_{j = 1}^{i-1} s_jg_j + t_jh_j.
  \end{equation}
  So as long as $t_i\leq \gamma_i$, $f_i^\prime$, defined by
  \begin{equation}
      f_i^\prime(x) = \displaystyle\sum_{j = 1}^{i-1} \langle s_jg_j + t_jh_j, x\rangle  + \alpha_{i}D_\phi(x,x_1) + (\gamma_i - t_i) G(x_i)+ \langle t_ih_i, x\rangle,
  \end{equation}
  is a strongly convex function with convexity parameter $\alpha_i\sigma$ (since $\phi$ is strongly convex with
  parameter $\sigma$), and $0\in \partial f^\prime_i(x_i)$. This implies that (see Lemma 6 in \cite{nesterov2009primal})
  \begin{equation}
      f^\prime_i(x_i) \leq f^\prime_i(z) - \frac{\alpha_i\sigma}{2}\|x_i - z\|_V^2.
  \end{equation}
  Adding $\langle s_ig_i, x_i - z\rangle$ to both sides gives
  \begin{equation}
      f^\prime_i(x_i) + \langle s_ig_i, x_i - z\rangle \leq f^\prime_i(z) - \left(\frac{\alpha_i\sigma}{2}\|x_i - z\|_V^2 - \langle s_ig_i, x_i - z\rangle\right).
  \end{equation}
  H\"older's inequality, combined with the Cauchy-Schwartz inequality implies that
  $$\left(\frac{\alpha_i\sigma}{2}\|x_i - z\|_V^2 - \langle s_ig_i, x_i - z\rangle\right) \geq -s_i^2\frac{1}{2\alpha_i\sigma}\|g_i\|_{V^*}^2,
  $$
  and we get (since $g_i\in \partial F(x_i)$ and so the Lipschitz assumption on $F$ implies $\|g_i\|_{V^*} \leq M$)
  \begin{equation}
      f^\prime_i(x_i) + \langle s_ig_i, x_i - z\rangle \leq f^\prime_i(z) + \frac{M^2}{2\alpha_i\sigma}s_i^2.
  \end{equation}
  Expanding the definition of $f^\prime_i$, we obtain
  \begin{equation}
  \begin{split}
      \sum_{j = 1}^i\langle s_jg_j + t_jh_j, x_i - z\rangle &\leq \alpha_{i}(D_\phi(z,x_1) - D_\phi(x_i,x_1))\\
      &+ (\gamma_i - t_i)(G(z) - G(x_i)) + \frac{M^2}{2\alpha_i\sigma}s_i^2.
  \end{split}
  \end{equation}
 Plugging this bound into \eqref{eq_rda_105} with $z = x_{i+1}$ and $z = x^*$, (recalling that 
 $\gamma_n = \sum_{i = 1}^{n-1} s_i - t_i$ and $\gamma_1 = t_1 = 0$), 
 we obtain
 \begin{equation}
 \begin{split}
  \displaystyle\sum_{i = 1}^n \langle s_ig_i + t_ih_i, x_i - x^*\rangle& \leq \alpha_n D_\phi(x^*,x_1)
  + \displaystyle\sum_{i = 2}^n (\alpha_{i-1} - \alpha_i)D_\phi(x_i,x_1) \\
  &+ \frac{M^2}{2\sigma}\displaystyle\sum_{i = 1}^n \frac{s_i^2}{\alpha_i}
  + \displaystyle\sum_{i = 2}^n (t_i - s_{i-1})G(x_i) + (\gamma_n - t_n)G(x^*).
  \end{split}
 \end{equation}
 Plugging this into \eqref{eq_rda_100} and rearranging terms, we obtain
 \begin{equation}
 \begin{split}
  \displaystyle\sum_{i = 1}^n s_i(f(x_i) - f(x^*)) &\leq \alpha_n D_\phi(x^*,x_1)
  + \frac{M^2}{2\sigma}\displaystyle\sum_{i = 1}^n \frac{s_i^2}{\alpha_i} \\
  &+ \displaystyle\sum_{i = 2}^n (\alpha_{i-1} - \alpha_i)D_\phi(x_i,x_1) \\
  &+ s_1 G(x_1) - \displaystyle\sum_{i = 2}^n (s_{i-1} - s_i)G(x_i) - s_nG(x^*).
  \end{split}
 \end{equation}
 Finally, we use the assumption that $\alpha_{i-1} - \alpha_i \leq 0$, that $s_{i-1} - s_i \geq 0$, and that
 $x_1\in \arg\min_x G(x)$ to conclude that
 $$\displaystyle\sum_{i = 2}^n (\alpha_{i-1} - \alpha_i)D_\phi(x_i,x_1) \leq 0,
 $$
 and that
 $$s_1 G(x_1) - \displaystyle\sum_{i = 2}^n (s_{i-1} - s_i)G(x_i) - s_nG(x^*) \leq s_1(G(x_1) - \min_x G(x)) = 0.
 $$
 This yields
 \begin{equation}
  \displaystyle\sum_{i = 1}^n s_i(f(x_i) - f(x^*)) \leq \alpha_n D_\phi(x^*,x_1)
  + \frac{M^2}{2\sigma}\displaystyle\sum_{i = 1}^n \frac{s_i^2}{\alpha_i},
 \end{equation}
 which completes the proof in the deterministic case.
 
 Now suppose that the subgradients are stochastic. Taking the expectation of equations \eqref{eq491} and \eqref{eq495}, we see that it suffices to prove that
 \begin{equation}
 \mathbb{E}\left(\displaystyle\sum_{i = 1}^n s_i(f(x_i) - f(x^*))\right) \leq 
 \left(\alpha_nD_\phi(x^*,x_1) + \frac{M^2}{2\sigma}\displaystyle\sum_{i = 1}^n \frac{s_i^2}{\alpha_i}\right).
 \end{equation}
Using the assumption that $\mathbb{E}_{\xi_i}(g(x_i,\xi_i))\in \partial f(x_i)$, the fact that
$x_i$ only depends on $\xi_1,...,\xi_{i-1}$, and the assumption that $t_i \geq 0$, we see that
\begin{equation}\label{eq602}
\begin{split}
 \mathbb{E}\left(\displaystyle\sum_{i = 1}^n s_i(f(x_i) - f(x^*))\right) &= \displaystyle\sum_{i = 1}^n \mathbb{E}_{\xi_1,...,\xi_{i-1}}(s_i(f(x_i) - f(x^*)))\\
&\leq \displaystyle\sum_{i = 1}^n \mathbb{E}_{\xi_1,...,\xi_{i-1}}(\langle s_i\mathbb{E}_{\xi_i}(g(x_i,\xi_i)) + t_ih_i, x_i - x^*\rangle) \\
  &+ \displaystyle\sum_{i = 1}^n \mathbb{E}_{\xi_1,...,\xi_{i-1}}((s_i - t_i)(G(x_i) - G(x^*))).
 \end{split}
 \end{equation}
Since $h_i$ and $x_i$ are independent of $\xi_i$, we can pull the $\mathbb{E}_{\xi_i}$ out of the
inner product above and the left hand side simply becomes
\begin{equation}\label{eq611}
 \mathbb{E}\left(\displaystyle\sum_{i = 1}^n \langle s_i(g(x_i,\xi_i)) + t_ih_i, x_i - x^*\rangle + 
 \displaystyle\sum_{i = 1}^n (s_i - t_i)(G(x_i) - G(x^*))\right).
\end{equation}
We now bound the term inside the expectation exactly as before to obtain
\begin{equation}
\begin{split}
 \displaystyle\sum_{i = 1}^n \langle s_i\mathbb{E}_{\xi_i}(g(x_i,\xi_i)) &+ t_ih_i, x_i - x^*\rangle + 
 \displaystyle\sum_{i = 1}^n (s_i - t_i)(G(x_i) - G(x^*)) \leq \\
 &\alpha_nD_\phi(x^*,x_1) + \frac{M^2}{2\sigma}\displaystyle\sum_{i = 1}^n \frac{s_i^2}{\alpha_i}.
 \end{split}
\end{equation}
Combining this with \eqref{eq611} and \eqref{eq602}, we get
 \begin{equation}
 \mathbb{E}\left(\displaystyle\sum_{i = 1}^n s_i(f(x_i) - f(x^*))\right) \leq 
 \left(\alpha_nD_\phi(x^*,x_1) + \frac{M^2}{2\sigma}\displaystyle\sum_{i = 1}^n \frac{s_i^2}{\alpha_i}\right),
 \end{equation}
 which completes the proof.

\end{proof}

\section{Experimental Results}
\label{experimental-results-section}
In this section, we provide experimental evidence that the XRDA method exhibits improved convergence behavior over RDA on some convex optimization problems. In particular, we consider the problem of calculating a sparse logistic regression on the MNIST dataset \cite{lecun1998gradient}. Specifically, the objective we consider is
\begin{equation}\label{logistic-regression-objective}
 L(W,b) = \frac{1}{N}\sum_{i=1}^N\left[ \log(\exp(y_i\cdot (Wx_i + b)) - \log\left(\sum_{j=1}^k \exp(e_j\cdot (Wx_i + b))\right)\right] + \lambda(|W|_1 + |b|_1),
\end{equation}
where $W\in \mathbb{R}^{k\times n}$, $b\in \mathbb{R}^k$ are the parameters of the logistic regression, $x_i\in \mathbb{R}^n$ is the training data and $y_i\in \mathbb{R}^k$ are the training labels \cite{james2013introduction}. For MNIST the specific parameters are $N = 60000$ (number of training data points), $n = 784$ (dimension of the image data), $k = 10$ (number of classes).

We optimize the objective \eqref{logistic-regression-objective} with $\lambda = 5\cdot 10^{-4}$ using stochastic forward backward-splitting, which is given by
\begin{equation}
 \begin{split}
  &\theta_{n+\frac{1}{2}} = \theta_n - s_ng_n\\
  &\theta_{n+1} = \arg\min_x \left(\frac{1}{2}\|\theta - \theta_{n+\frac{1}{2}}\|^2 + s_n\lambda|\theta|_1\right).
 \end{split}
\end{equation}
Here $\theta_n = (W_n,b_n)$ are the parameters to be optimized.
Next, we optimize the same objective using forward-backward RDA, given by
\begin{equation}
 \begin{split}
  &\theta_{n+\frac{1}{2}} = \theta_{n-\frac{1}{2}} - s_ng_n\\
  &\theta_{n+1} = \arg\min_x \left(\frac{1}{2}\|\theta - \theta_{n+\frac{1}{2}}\|^2 + \left(\sum_{i=1}^n s_n\right)\lambda|\theta|_1\right).
 \end{split}
\end{equation}
Finally, we optimize the objective using XRDA which is set to asymptotically produce a fixed backward step-size of $M$. This iteration is given by
\begin{equation}
 \begin{split}
  &\mu_n = 1-\frac{s_n}{M}\\
  &\theta_{n+\frac{1}{2}} = \mu_n\theta_{n-\frac{1}{2}} + (1-\mu_n)\theta_n - s_ng_n\\
  &b_n = \mu_nb_{n-1} + s_n\\
  &\theta_{n+1} = \arg\min_x \left(\frac{1}{2}\|\theta - \theta_{n+\frac{1}{2}}\|^2 + b_n\lambda|\theta|_1\right).
 \end{split}
\end{equation}
Note that this corresponds to the XRDA iteration \ref{xrda-method} with $\alpha_n = 1$, $\mu_n = 1-s_n/M$ and $\phi = \frac{1}{2}\|\cdot\|^2$. We take $M=500,1000,2500,5000,$ and $10000$. We run all of these methods with an initial iterate $\theta_1 = \theta_{1/2} = 0$, a batch size of $10$, and step size $s_n = 3.0\cdot n^{-\frac{1}{2}}$ for a total of $50$ epochs. Note that this scaling of the step size results in the optimal theoretical convergence rate of $O(n^{-\frac{1}{2}})$ for all methods. We plot the number of non-zero components, the test accuracy and the training loss of each of these methods as a function of iteration in Figure \ref{experimental-results}.
\begin{figure}\label{experimental-results}
\begin{center}
 \includegraphics[scale=0.5]{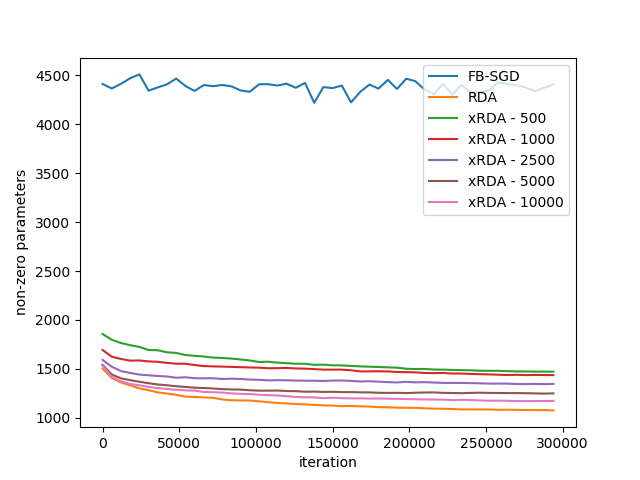}
 \includegraphics[scale=0.5]{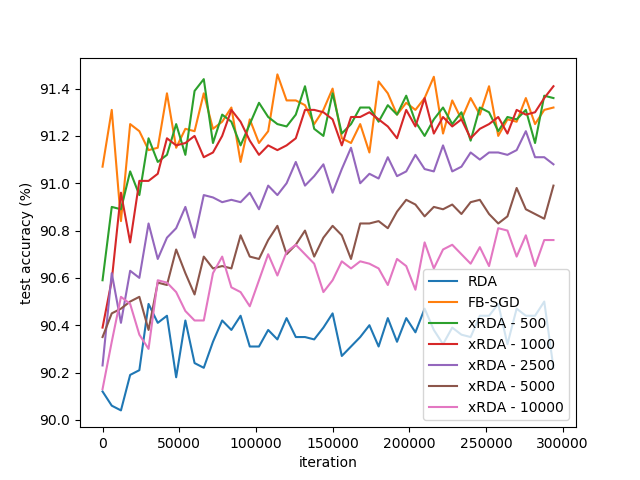}
 \includegraphics[scale=0.6]{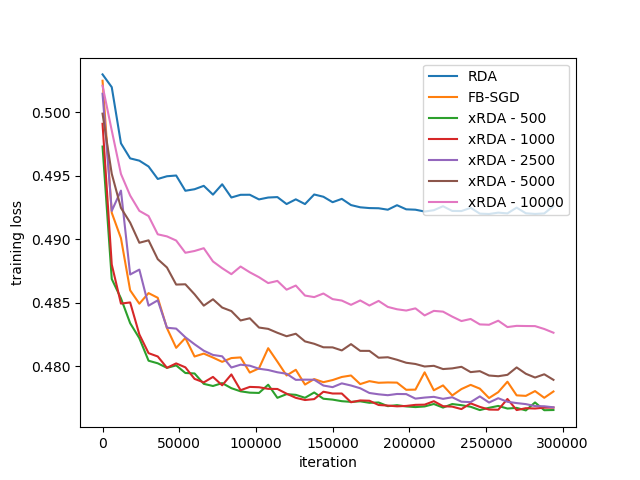}
 \end{center}
 \caption{Comparison of stochastic forward-backward splitting (FB-SGD), RDA, and XRDA with various values for the asymptotic backward step size $M$. We have plotted the number of non-zero parameters, the test accuracy and the training loss for each of the methods.}
\end{figure}

From this experiment, we draw the following conclusions. First, we see that forward-backward stochastic gradient descent doesn't generate sparse solutions, as expected. In contrast, RDA and XRDA all generate sparse solutions. Further, the larger the backward step size $M$ used in XRDA, the sparser the iterates are. This extends to RDA which generates the sparsest iterates. However, the difference between the sparsity generated by these methods is not particularly large. 

Next, considering the training loss (i.e. the objective \eqref{logistic-regression-objective} we are trying to minimize), we see that the larger backward step size used in XRDA improves the convergence rate over forward-backward SGD. In particular, the loss converges fastest with asymptotic backward step size $M=500$ and $M=1000$. With an asymptotic backward step size of $M=2500$, XRDA converges at the roughly the same rate as forward-backward SGD initially and then converges faster for larger iterations. However, once the backward step size of XRDA is pushed beyond this point, the convergence rate slows dramatically as can be seen for $M=5000$ and $M=10000$. Finally, RDA exhibits by far the slowest convergence rate. This shows that the large backward step size produced by the RDA iteration can dramatically slow down convergence. The best results are obtained using XRDA with parameters chosen to give an appropriate backward step size.

\section{Conclusion}
We have derived and analysed the extended regularized dual averaging (XRDA) method in Table \ref{xrda-method}, 
a class of methods
for solving regularized optimization problems which generalize RDA \cite{xiao2010dual}. We derived and analysed the convergence of 
these methods for regularized convex optimization problems. The main novelty of these new methods is that we can control the proximal (or backward) step size. Finally, we have demonstrated experimentally that controlling the backward step size can significantly improve the convergence of the resulting methods while still preserving the desired structure (such as sparsity) of the iterates. 

We also remark that these methods are useful for non-convex optimization problems appearing in machine learning as well, since BinaryConnect \cite{courbariaux2015binaryconnect} and BinaryRelax \cite{yin2018binaryrelax} are special cases of the method applied to non-convex neural networks trianing, RDA has successfully been applied to non-convex problems such as the training of sparse neural networks \cite{jia2018modified}, and the XRDA method has also been shown to effectively train sparse neural networks \cite{siegel2020training}.

\section{Acknowledgements}
We would like to thank Dr. Xiao Lin for helpful comments and discussion during the preparation of this work. This work was supported by the Verne M. Willaman Fund, the Center for Computational Mathematics and Applications (CCMA) at the Pennsylvania State University, and the National Science Foundation (Grant No. DMS-1819157).

\bibliographystyle{spmpsci}      %
\bibliography{refs}   %

\end{document}